\documentclass{amsart}
\usepackage{amssymb}

\usepackage{comment}
\usepackage{color}
\usepackage{bm}
\newtheorem{theorem}{Theorem}[section]
\newtheorem{lemma}[theorem]{Lemma}

\theoremstyle{definition}
\newtheorem{definition}[theorem]{Definition}

\newtheorem{question}[theorem]{Question}
\theoremstyle{remark}

\numberwithin{equation}{section}

\begin{document}

\title[Complete orbit equivalence relations]{Complete orbit equivalence relation and non-universal Polish groups}



\author{Longyun Ding}
\address{Nankai university}
\curraddr{School of Mathematical Sciences and LPMC, Nankai University, Tianjin 300071, P.R. China}
\email{dingly@nankai.edu.cn}

\author{Ruiwen Li}
\address{Nankai university}
\curraddr{School of Mathematical Sciences and LPMC, Nankai University, Tianjin 300071, P.R. China}
\email{rwli@mail.nankai.edu.cn}

\author{Bo Peng}
\address{McGill University}
\curraddr{Department of Mathmatics and Statistics, McGill University. 805 Sherbrooke Street West Montreal, Quebec, Canada, H3A 2K6}
\email{bo.peng3@mail.mcgill.ca}
\thanks{}


\date{}

\dedicatory{}

\commby{}

\begin{abstract}
We show that 
 a non-universal Polish group can induce a complete orbit equivalence relation, which answers a question of Sabok from \cite{OPENPROBLEMS}. 
\end{abstract}

\maketitle

\section{Introduction}

Let $G$ be a Polish group acting on a Polish space in a Borel way. Let $E^X_G$ be the \textbf{orbit equivalence relation} induced by the action. In other words, for $x,y \in X$,
$$
xE^X_Gy\,\,\Leftrightarrow\,\,\exists g\in G\,\,gx=y.
$$

 Let $E$ and $F$ be two equivalence relations on Polish spaces $X$ and $Y$, respectively. A Borel function from $X$ to $Y$ is called a \textbf{Borel reduction} from $E$ to $F$, if for any $x_1,x_2\in X$, we have
$$
x_1Ex_2\,\,\mbox{if and only if}\,\,f(x_1)Ff(x_2).
$$
We say $E$ is \textbf{Borel reducible} to $F$ if there is a Borel reduction from $E$ to $F$ and denote by $E\leq_B F$. 

It is natural to ask about the connection between the structure of a Polish group $G$ and the complexity of the equivalence relations it induces. The connection is well-studied if $G$ is countable. For example, if $G=\mathbb{Z}$, then all actions induced by $G$ are hyperfinite. In an important breakthrough, Gao and Jackson \cite{Gaojackson} proved that all actions induced by countable Abelian groups are hyperfinite.  Weiss conjectured that hyperfiniteness holds for all countable amenable group actions. This conjecture is still open so far.  

For every Polish group $G$, Becker and Kechris \cite[Theorem 3.3.4]{Gaobook} proved that there is a most complicated orbit equivalence relation induced by $G$, called a \textbf{maximal $G$-action}. In other words, there exists a $G$-action such that all $G$-actions are Borel reducible to it. Thus, the question regarding connections between a Polish group and its actions can be restated as the following: given a Polish group $G$, where is its maximal action located in the complexity hierarchy?

A Polish group $G$ is \textbf{universal} if all Polish groups can be continuously embedded as closed subgroup of $G$. An equivalence relation $E$ is \textbf{complete} in a class $\mathcal{C}$ of equivalence relations if $E$ is in $\mathcal{C}$ and all equivalence relations in $\mathcal{C}$ are Borel reducible to $E$.  Complete orbit equivalence relations exist. Indeed, the maximal action of any universal Polish group is a complete orbit equivalence relation by the Mackey-Hjorth extension theorem \cite[Theorem 3.5.2]{Gaobook}.  There are also natural examples in topology and $C^*$-algebras (see, e.g. \cite{Sabokcomplete}, \cite{Zielinski}).

It is natural to ask that if there is a connection between completeness of an orbit equivalence relation and the structure of the group which induces this equivalence relation. 

 In this paper, we address the following question posed by Sabok in \cite{OPENPROBLEMS}:

\begin{question} \label{main}
    (\cite{OPENPROBLEMS}) Does there exist a non-universal Polish group which induces a complete orbit equivalence relation?
\end{question}

We answer this question in the positive,

\begin{theorem}
    There exists a non-universal Polish group which induces complete orbit equivalence relation.
\end{theorem}

In \cite{Ding}, Ding found an example of a surjectively universal Polish group. We note that such group can induce complete orbit equivalence relation, and we show that this group is not universal. As a matter of fact, in \cite{Saboksurvey}, Sabok mentioned without proof that a surjectively universal but not universal Polish group could serve as an example for Question \footnote{In fact, the authors found out about this after this work has been completed}. However, showing that a surjectively universal group which is not universal exists requires a proof and in the subsequent survey \cite{OPENPROBLEMS}  Question \ref{main} was stated as open.

\section{Preliminaries}
In this section we discuss the construction and some results about surjectively universal Polish groups from \cite{Ding} and \cite{DingGao}.

For a nonempty set $X$, let $X^{-1} =\{x^{-1}: x \in X\}$ be a disjoint copy of $X$, take $e\notin X\cup X^{-1}$. Denote $X\cup X^{-1}\cup\{e\}$ by $\overline{X}$. For $x\in X$, by $(x^{-1})^{-1}$ we mean $x$, and by $e^{-1}$ we mean $e$.

 We denote the set of nonempty words on $\overline{X}$ by $W(X)$, in other words $W(X)=\overline{X}^{<\omega}\setminus\{\emptyset\}$. For $w\in W(X)$, $|w|$ is the length of $w$. A word $w\in W(X)$ is \textbf{irreducible} if $w=e$ or $w=x_0\cdots x_n$ such that $n\ge0$, $x_i\ne e$ for every $0\le i\le n$ and $x_i\ne x_{i+1}^{-1}$ for every $0\le i\le n-1$. We denote the set of irreducible words in $W(X)$ by $F(X)$.  For a word $w\in W(X)$,
the \textbf{reduced word} for $w$ is the irreducible word obtained by successively replacing any occurrence of $xx^{-1}$ in $w$ by $e$ where $x\in \overline{X}$, and eliminating $e$ from any occurrence of the form $w_1ew_2$, where at least one of $w_1$ and $w_2$ is nonempty. Note that the reduced word of $w\in W(X)$ is independent from the order of eliminating $e$ or replacing $xx^{-1}$ to $e$, denote the reduced word for $w\in W(X)$ by $w'$, then $w'\in F(X)$. For $u,v\in F(X)$, let $u\cdot v=(u^\smallfrown v)'$, then $(F(X),\cdot)$ is a group, named the \textbf{free group} on $X$. The neutral element of $F(X)$ is $e$.

To define a metric on $F(X)$, firstly we assume that $X$ is a metric space with a metric $d\le1$, that is $d(x,y)\le1$ for every $x,y\in X$. We can extend $d$ to $\overline{X}$ by defining $$d(x^{-1},  y^{-1})=d(x,y),\quad d(x^{-1},y)=d(x,e)=d(x^{-1},e)=1$$for every $x,y\in X$. Then to get a metric on $F(X)$ we need some more concepts.
\begin{definition}\!\!\mbox{(Ding-Gao, \cite{DingGao_1}, \cite[Definition 2.3.2]{Gaobook})}
For $m,n\in\mathbb{N}$ and $m\le n$, a bijection $\theta$ on $\{m,\ldots,n\}$ is a \textbf{match} if
\begin{enumerate}
    \item[(1)] $\theta^2={\rm id}$;
    \item[(2)] there is no $i,j\in\{m,\ldots,n\}$ such that $i<j<\theta(i)<\theta(j)$.
\end{enumerate}
\end{definition}

\begin{definition}[Ding-Gao, \cite{DingGao}]
Let $\mathbb{R}_+$ be the set of non-negative real numbers. A function $\Gamma : \overline{X}\times\mathbb{R}_+\rightarrow\mathbb{R}_+$ is a \textbf{scale} on $\overline{X}$ if the following hold for any $x\in\overline{X}$ and $r\in\mathbb{R}_+$
\begin{enumerate}
\item[(i)] $\Gamma(e,r)=r$, $\Gamma(x,r)\ge r$;
 \item[(ii)] $\Gamma(x,r)=0$ iff $r=0$;
 \item[(iii)] $\Gamma(x,\cdot)$ is a monotone increasing function with respect to the second variable;
 \item[(iv)] ${\rm lim}_{r\rightarrow0} \Gamma(x,r)=0$.
\end{enumerate}
\end{definition}

If we take $\Gamma(x,r)=r$ for every $x\in\overline{X}$ and $r\in\mathbb{R}_+$, then $\Gamma$ is a scale on $\overline{X}$, we call it the \textbf{trivial} scale on $\overline{X}$.

Let $G$ be a topological group with compatible left-invariant metric $d_G$. Define $\Gamma_G:G\times\mathbb{R}_+ \rightarrow \mathbb{R}_+$ by
 $$\Gamma_G(g,r) ={\rm max}\{ r,{\rm sup}\{d_G(1_G,g^{-1}hg):d_G(1_G,h)\le r\}\} .$$
 It is easy to see that $\Gamma_G$ satisfies the conditions (i)–(iii) in the definition of a scale, and (iv) is from the compatibility of $d_G$. We will also call $\Gamma_G$ \textbf{the scale on $G$}.

    Then given a scale $\Gamma$ on $\overline{X}$, we will define a metric on $F(X)$ related to $\Gamma$.

\begin{definition}[Ding-Gao, \cite{DingGao}]
     Let $\Gamma$ be a scale on $\overline{X}$. For $l\in\mathbb{N}$, $w\in W(X)$ with $|w|=l+1$ and $\theta$ a match on $\{0,\ldots,l\}$, define $N^{\theta}_\Gamma (w)$ by induction on $l$ as follows:
\begin{enumerate}
 \item[(0)] for $l =0$, let $w =x$ and define $N^\theta_\Gamma(w)=d(e,x)$;
\item[(1)] if $l>0$ and $\theta(0)=k<l$, let $\theta_1 =\theta \upharpoonright\{0,\ldots,k\}$, $\theta_2=\theta \upharpoonright\{k+1,\ldots,l\}$ and $w =w_1^\smallfrown w_2$ where $|w_1|=k +1$; define
 $$N^{\theta}_\Gamma(w)=N^{\theta_1}_\Gamma(w_1)+N^{\theta_2}_\Gamma(w_2);$$

\item[(2)] if $l>0$ and $\theta(0)=l$, let $\theta_1 =\theta \upharpoonright\{1,\ldots,l-1\}$ and $w = x^{-1}w_1y$ where $x,y \in\overline{X}$; then $|w_1| =l-1$. Define
 $$N^{\theta}_\Gamma(w)=d(x,y)+{\rm max}\{\Gamma(x,N^{\theta_1}_\Gamma(w_1)),\Gamma(y,N^{\theta_1}_\Gamma(w_1))\}.$$
\end{enumerate}

Then we take $$N_\Gamma(w)={\rm inf}\{N^\theta_\Gamma(w^*):(w^*)'=w,\,\theta\;is\;a\;match\}$$ for $w\in F(X)$, and we take $$\delta_\Gamma(u,v)=N_\Gamma(u^{-1}v)$$ for $u,v\in F(X)$.
\end{definition}

For a metric space $(X,d)$ and the trivial scale $\Gamma_0$ on $\overline{X}$, we call $\delta_{\Gamma_0}$ \textbf{the Graev metric} on $F(X)$. 

For $u=x_0\cdots x_n$ and $v=y_0\cdots y_n$ in $W(X)$, let $$\rho(u,v)=\Sigma_{0\le i\le n}d(x_i,y_i).$$ Given $w=x_0\cdots x_n\in W(X)$ and a match $\theta$ on $\{0,\cdots,n\}$, let 
$$x^\theta_i= \left \{\begin{array}{lr}  x_i,&\theta(i)> i.\\ e,& \theta(i)=i.\\ x^{-1}_{\theta(i)},&\theta(i)< i. \end{array} \right.$$and $w^\theta=x^\theta_0\cdots x^\theta_n$. Then we have

\begin{theorem}\label{calculate}\!\!\mbox{{\rm (Ding-Gao, \cite[Theorem 2.6.5]{Gaobook})}}
Let $(X,d)$ be a metric space, $\Gamma_0$ is the trivial scale on $\overline{X}$, then $$\delta_{\Gamma_0}(w,e)={\rm min}\{\rho(w,w^\theta):\;\theta\, is\;a\;match\}$$ for every $w\in F(X)$.
\end{theorem}

Then we have the following theorem in \cite{DingGao}
\begin{theorem}\!\!\mbox{{\rm (Ding-Gao, \cite[Theorem 3.9]{DingGao_1})}}
Let $(X,d)$ be a metric space, $\Gamma$ is a scale on $\overline{X}$, then the $d_\Gamma$ is a left invariant
 metric on $F(X)$ extending $d$. Furthermore, $F(X)$ is a topological group under the topology induced by $\delta_\Gamma$.
\end{theorem}

Denote the topological group $F(X)$ with the metric $\delta_\Gamma$ by $F_\Gamma(X)$. Let $\delta^{-1}_\Gamma(u,v)=\delta_\Gamma(u^{-1},v^{-1})$ for $u,v\in F(X)$, and $d_\Gamma=\delta^{-1}_\Gamma+\delta_\Gamma$, then $d_\Gamma$ is a compatible metric on topological group $F_\Gamma(X)$. The completion of $(F_\Gamma(X),d_\Gamma)$ is a Polish group, denoted by $\overline{F}_\Gamma(X)$. Then the main theorem in \cite{Ding} says that such a group of this form can be a surjectively universal group.

We denote the Baire space $\omega^\omega$ by $\mathcal{N}$, $d$ is the classical metric on $\mathcal{N}$ such that $d(x,y)={\rm max}\{2^{-n}:x(n)\ne y(n)\}$. For $n\in\mathbb{N}$, let $$\mathcal{N}_n=\{x\in\mathcal{N}:\forall m\ge n,x(m)=0\},$$ and $\mathcal{N}_\omega=\bigcup_{n\in\mathbb{N} }\mathcal{N}_n$. For $x\in\mathcal{N}$, let 
\[\pi_n(x)(m)=
    \left\{\begin{array}{lr}
    x(m),&m<n.\\
    0,& m\ge n.
    \end{array}
    \right.\] Then a scale $\Gamma$ on $\overline{\mathcal{N}}$ is \textbf{regular} if for every $x\in\overline{\mathcal{N}}$, $r\in\mathbb{R}_+$ and $n\in\mathbb{N}$, we have $\Gamma(x,r)\ge\Gamma(\pi_n(x),r)$.

    \begin{theorem}\label{surjectively universal}\!\!\mbox{{\rm (Ding-Gao, \cite[Example 4.9]{Ding})}}
There is a regular scale $\Gamma$ on $\overline{\mathcal{N}}$ such that $\overline{F}_\Gamma(\mathcal{N})$ is a surjectively universal Polish group.
    \end{theorem}

    The next theorem in \cite{DingGao} tells us how to get a continuous homomorphism from $\overline{F}_\Gamma(X)$ to another topological group.

\begin{theorem}\label{main lemma}\!\!\mbox{{\rm (Ding-Gao, \cite[Lemma 3.7]{DingGao_1})}}
    Let $G$ be a topological group and $d_G$ a compatible left-invariant metric on $G$. Let $\Gamma$ be a scale on $\overline{X}$. Let $\varphi:X\rightarrow G$ be a function. Suppose that for any $x,y\in \overline{X}$ and $r\in\mathbb{R}_+$:
\begin{enumerate}
    \item[(1)] $\varphi(e)=1_G$; $\varphi(x^{-1})=\varphi(x)^{-1}$;
 \item[ (2)] $d_G(\varphi(x),\varphi(y)) \le d(x,y)$; and
 \item[ (3)] $\Gamma_G(\varphi(x),r) \le\Gamma(x,r)$.
\end{enumerate}
 Then $\varphi$ can be uniquely extended to a continuous group homomorphism $\Phi:F(X)\rightarrow G$ such that for any $w \in F(X)$
 $$d_G(\Phi(w),1_G)\le N_\Gamma(w)$$
 \end{theorem}

 \section{Proof of the Main Theorem}

 By \cite{Ding}, we can fix a regular scale $\Gamma$ on $\overline{\mathcal{N}}$ such that $\overline{F}_\Gamma(\mathcal{N})$ is a surjectively universal Polish group, and $\Gamma(x,\cdot)$ is continuous for every $x\in\overline{\mathcal{N}}$. Recall $d_\Gamma=\delta_\Gamma+\delta^{-1}_\Gamma$ on ${F}_\Gamma(\mathcal{N})$, by abusing the notation we also denote the extension of it on $\overline{F}_\Gamma(\mathcal{N})$ by $d_\Gamma$.
\begin{lemma}
Every surjectively universal Polish group induces a complete orbit equivalence relation.
\end{lemma}
\begin{proof}
Let $E$ be a complete orbit equivalence relation induced by a continuous group action of $H$ on $X$, where $H$ is a Polish group and $X$ is a Polish space. For every surjectively universal Polish group $G$, take $\pi:G\rightarrow X$ be a continuous surjective homomorphism. For $g\in G$ and $x\in X$, let $g\cdot x=\pi(g)\cdot x$, this defines a continuous group action of $G$ on $X$ that induces $E$.
\end{proof}

In the rest of this section we will show the group in Theorem \ref{surjectively universal} is not universal. It is mentioned in \cite{Ding} that the group is abstractly isomorphic to a $\bm{\Pi}^0_3$ subgroup of $S_\infty$, but we will verify that it is continuously embedded into a totally disconnected topological group.

\begin{lemma}\label{3.2}
For every $n\in\mathbb{N}$, $F(\mathcal{N}_n)$ is a discrete closed subgroup of $\overline{F}_\Gamma(\mathcal{N})$.
\end{lemma}
\begin{proof}
Let $\Gamma_0$ be the trivial scale on $\overline{\mathcal{N}_n}$.  Note that for every $x\ne y\in\mathcal{N}_n$, we have $d(x,y)\ge2^{-n}$. Thus, for $u\ne v\in F(\mathcal{N}_n)$, we have $u^{-1}v\ne e$ and consequently $\rho(u^{-1}v,(u^{-1}v)^\theta)\ge2^{-n}$ for every match $\theta$. Then $\delta_{\Gamma_0}(u,v)=\delta_{\Gamma_0}(u^{-1}v,e)\ge2^{-n}$ by Theorem \ref{calculate}. By \cite[Lemma 3.6 (i)]{DingGao}, $\delta_\Gamma(u,v)\ge\delta_{\Gamma_0}(u,v)$. So $d_\Gamma(u,v)\ge 2^{-n}$ for every $u\ne v\in F(\mathcal{N}_n)$, $F(\mathcal{N}_n)$ is a discrete closed subgroup of $\overline{F}_\Gamma(\mathcal{N})$.

\end{proof}

We denote the topological group $F(\mathcal{N}_n)$ with the metric $\delta_\Gamma$ by $F_\Gamma(\mathcal{N}_n)$. Then $F_\Gamma(\mathcal{N}_n)$ is a discrete topological space by Lemma \ref{3.2}.
 
\begin{lemma}\label{gamma}
For every $w\in\overline{\mathcal{N}}$ and every $r\in\mathbb{R}_+$, we have $\Gamma_G(w,r)\le\Gamma(w,r)$ where $G$ is the group ${F}_\Gamma(\mathcal{N})$ together the metric $\delta_\Gamma$.
\end{lemma}
\begin{proof}
By definition, $$\Gamma_G(w,r)={\rm max}\{r,{\rm sup}\{N_\Gamma(w^{-1}uw):N_\Gamma(u)\le r\}\}.$$ We need to show that for every $u\in {F}_\Gamma(\mathcal{N})$ with $N_\Gamma(u)\le r$, we have $N_\Gamma(w^{-1}uw)\le \Gamma(w,r)$. For every $\epsilon> 0$, by the definition of $N_\Gamma (u)$, there is $v\in W(\mathcal{N})$ and a match $\theta$ on $\{0,\cdots,|v|-1 \}$ such that $v'=u$ and $N^{\theta}_\Gamma(v)\le r+\epsilon$. Let $\eta$ be a match on $\{0,\cdots,|v|+1 \}$ such that $\eta(0)=|v|+1$ and $\eta(i)=\theta(i-1)+1$ for $1\le i\le|v|$. Then by definition of $N^\eta_\Gamma(w^{-1}vw)$ we have $$N_\Gamma(w^{-1}vw)\le N^\eta_\Gamma(w^{-1}vw)=\Gamma(w,N^{\theta}_\Gamma(v))\le \Gamma(w,r+\epsilon).$$When $\epsilon$ goes to $0$, by continuity of $\Gamma(w,\cdot)$, we have $N_\Gamma(w^{-1}vw)\le\Gamma(w,r)$. Since we also have $r\le\Gamma(w,r)$, we get $\Gamma_G(w,r)\le\Gamma(w,r)$.
\end{proof}
\begin{lemma}\label{extension}
For every $n\in\mathbb{N}$, the map $\pi_n$ can be uniquely extended to a continuous group homomorphism $f_n$ from $\overline{F}_\Gamma(\mathcal{N})$ to ${F}_\Gamma(\mathcal{N}_n)$ such that $d(g,h)\ge d(f_n(g),f_n(h)) $ for every $g,h\in \overline{F}_\Gamma(\mathcal{N})$.
\end{lemma}
\begin{proof}
Firstly we show that $\pi_n$ can be extended to a group homeomorphism $f_n$ from ${F}_\Gamma(\mathcal{N})$ to ${F}_\Gamma(\mathcal{N}_n)$ such that $N_\Gamma(u)\ge N_\Gamma(f_n(u))$ for every $u\in {F}_\Gamma(\mathcal{N})$. Let $\pi_n(e)=e$ and $\pi_n(x^{-1})=(\pi_n(x))^{-1}$. We will use Lemma \ref{main lemma} to extend $\pi_n$ to a homomorphism. Items (1) and (2) of Lemma \ref{main lemma} are by definition, and item (3) follows from Lemma \ref{gamma}. Then we get such an extension $f_n$. 

Next we show that $f_n$ can be extended to $\overline{F}_\Gamma(\mathcal{N})$. The group $\overline{F}_\Gamma(\mathcal{N})$ is the completion of ${F}_\Gamma(\mathcal{N})$ with the metric $d_\Gamma$ such that for every $u,v\in {F}_\Gamma(\mathcal{N})$, $$d_\Gamma (u,v)=\delta_\Gamma(u,v)+\delta_\Gamma(u^{-1},v^{-1}),$$ and ${F}_\Gamma(\mathcal{N}_n)$ is a closed subgroup of $\overline{F}_\Gamma(\mathcal{N})$. We have that 

\begin{equation}\label{1}
    d_\Gamma(u,v)=N_\Gamma(u^{-1}v)+N_{\Gamma}(uv^{-1})\ge N_\Gamma(f_n(u^{-1}v))+N_{\Gamma}(f_n(uv^{-1}))\ge d_\Gamma(f_n(u),f_n(v))
\end{equation}  for every $u,v\in {F}_\Gamma(\mathcal{N})$. If $(x_m)_{m\in\mathbb{N}}$ is a Cauchy sequence in ${F}_\Gamma(\mathcal{N})$ under the metric $d_\Gamma$, then $(f_n(x_m))_{m\in\mathbb{N}}$ is a Cauchy sequence in ${F}_\Gamma(\mathcal{N}_n)$ under the metric $d_\Gamma$. So $f_n$ can be uniquely extended to a continuous group homomorphism from $\overline{F}_\Gamma(\mathcal{N})$ to ${F}_\Gamma(\mathcal{N}_n)$. 

Finally we argue that for this extension $f_n$, we have $d(g,h)\ge d(f_n(g),f_n(h)) $ for every $g,h\in \overline{F}_\Gamma(\mathcal{N})$. By (\ref{1}), we know that $d(g,h)\ge d(f_n(g),f_n(h)) $ for every $g,h\in {F}_\Gamma(\mathcal{N})$. Since ${F}_\Gamma(\mathcal{N})$ is dense in $\overline{F}_\Gamma(\mathcal{N})$ and $f_n$ is continuous, we get that $d(g,h)\ge d(f_n(g),f_n(h)) $ holds for every $g,h\in \overline{F}_\Gamma(\mathcal{N})$
\end{proof}

\begin{theorem}\label{3.5}
There is a continuous embedding from $\overline{F}_\Gamma(\mathcal{N})$ to $\prod_{n\in\mathbb{N}}F(\mathcal{N}_n)$.
\end{theorem}
\begin{proof}
    Let $f$ from $\overline{F}_\Gamma(\mathcal{N})$ to $\prod_{n\in\mathbb{N}}F(\mathcal{N}_n)$ map $g$ to $(f_n(g))_{n\in\mathbb{N}}$ for every $g\in \overline{F}_\Gamma(\mathcal{N})$. The map $f$ is a continuous homomorphism, we just need to show it is an injection.

    For $g\ne h\in \overline{F}_\Gamma(\mathcal{N})$, let $0<\epsilon< d_\Gamma(g,h)/4$. Since $F(\mathcal{N}_\omega)$ is dense in $\overline{F}_\Gamma(\mathcal{N})$, we can take $u,v\in F(\mathcal{N}_\omega)$ such that $d_\Gamma(u,g)\le\epsilon$ and $d_\Gamma(v,h)\le\epsilon$. There is $n\in\mathbb{N}$ such that $u,v\in F(\mathcal{N}_n)$, and since $f_n$ entends $\pi_n$, we have $f_n(u)=u,f_n(v)=v$. By Lemma \ref{extension}, we have $d_\Gamma(f_n(g),u)\le\epsilon$ and $d_\Gamma(f_n(h),v)\le\epsilon$, so $$d_\Gamma(f_n(g),f_n(h))\ge d_\Gamma(u,v)-2\epsilon\ge d_\Gamma(g,h)-4\epsilon>0.$$Then $f_n(g)\ne f_n(h)$, so $f(g)\ne f(h)$, and hence $f$ is an injection.
\end{proof}

\begin{theorem}
The group $\overline{F}_\Gamma(\mathcal{N})$ is not a universal Polish group.
\end{theorem}
\begin{proof}
We argue that $(\mathbb{R,+})$ is not isomorphic to a closed subgroup of $\overline{F}_\Gamma(\mathcal{N})$.  If $(\mathbb{R,+})$ is isomorphic to a closed subgroup of $\overline{F}_\Gamma(\mathcal{N})$, then by Theorem \ref{3.5} there is a continuous embedding from $(\mathbb{R,+})$ to $\prod_{n\in\mathbb{N}}F(\mathcal{N}_n)$. The topological space $\mathbb{R}$ is connected but $\prod_{n\in\mathbb{N}}F(\mathcal{N}_n)$ is totally disconnected, a contradiction.

We also have the alternative proof which does use Theorem \ref{3.5}. As mentioned in \cite{Ding}, $\overline{F}_\Gamma(\mathcal{N})$ is an abstract subgroup of $S_\infty$. The isometry group of the Urysohn space ${\rm Iso}(\mathbb{U})$ cannot be continuously embedded to $\overline{F}_\Gamma(\mathcal{N})$, because by automatic continuity \cite{Sabok,Sabokerr} (see \cite{mILIN,rose} for new proofs), ${\rm Iso}( \mathbb{U})$ cannot be embedded to $S_\infty$ as an abstract subgroup since ${\rm Iso}(\mathbb{U})$ is connected and $S_\infty$ is totally disconnected.
\end{proof}


\bibliographystyle{plain}
\bibliography{bibliography}
\end{document}